\documentclass[12pt]{article}  
\usepackage{amsmath}
\usepackage{theorem}                 
\usepackage{amssymb}
\usepackage{xypic}

\newtheorem{theorem}{Theorem}

\newtheorem{proposition}{Proposition}

\newenvironment{definition}
{\smallskip\noindent{\bf Definition\/}:}{\smallskip\par}

\newenvironment{example}
{\smallskip\noindent{\bf Example\/}.}{\smallskip\par}
\newenvironment{examples}
{\smallskip\noindent{\bf Examples\/}.}{\smallskip\par}
\newenvironment{remark}
{\smallskip\noindent{\bf Remark\/}.}{\smallskip\par}

\newenvironment{proof}
{\noindent{\it Proof\/}.}{{ \hfill $\Box$}\smallskip\par}

\newcommand{\CC}{{\Bbb C}}

\newcommand{\PP}{{\Bbb P}}

\newcommand{\RR}{{\Bbb R}}
\newcommand{\ZZ}{{\Bbb Z}}

\newcommand{\calO}{{\cal O}}

\newcommand{\calD}{{\cal D}}

\newcommand{\uu}{\underline{u}}
\newcommand{\vv}{\underline{v}}
\newcommand{\ww}{\underline{w}}
\newcommand{\xx}{\underline{x}}
\newcommand{\kk}{\underline{k}}
\newcommand{\ttt}{\underline{t}}
\newcommand{\mm}{\underline{m}}
\newcommand{\uone}{\underline{1}}
\newcommand{\uzero}{\underline{0}}

\title{Multi-variable Poincar\'e series associated with Newton diagrams}
\author{W.~Ebeling and S.~M.~Gusein-Zade
\thanks{Partially supported by the DFG (Eb 102/6--1), RFBR--07--01--00593,
NSh-709.2008.1.
}
}
\date{}

\begin{document}

\maketitle

\begin{abstract}
We define a multi-index filtration on the ring of germs of functions on a hypersurface singularity associated with its Newton diagram and compute the multivariable Poincar\'e series of this filtration in some cases.
\end{abstract}

\section*{Introduction}
Poincar\'e series of filtrations (including multi-index ones) on the ring of germs of functions on a complex analytic variety are of interest for some problems (see, e.g., \cite{CDG2}, \cite{CHR}, \cite{Ne}, \dots). 
In a number of cases they have an A'Campo type form being products/ratios of binomials of the form $(1-\ttt^{\, \mm})$. In some cases they are connected with monodromy zeta functions corresponding to the singularity (see, e.g., \cite{CDG1, EG}).

For quasi-homogeneous singularities one has the classical Poincar\'e series in one variable. A Poincar\'e series of one variable also corresponds to the semigroup of values of an irreducible curve singularity.
The initial motivation to consider multi-variable Poincar\'e series stems from the study of reducible curve singularities \cite{CDG2}. Recently they were found to be connected with the study of Seiberg-Witten invariants for surface singularities \cite{Ne}.

We define a multi-index filtration on the ring of germs of functions on a hypersurface singularity associated with its Newton diagram. One can say that this filtration is a multi-index generalization of the quasi-homogeneous one. We compute the Poincar\'e series of this filtration for curve singularities and for some singularities of more variables. In the computed cases, they turn out to be of A'Campo type.

 \section{Multi-index filtrations and their Poincar\'e series}
A function $v$ on the ring $\calO_{X,0}$ of germs of functions on the germ $(X,0)$ of an analytic space with values in $\ZZ_{\geq 0} \cup \{ \infty \}$ is called a {\em valuation} if
\begin{itemize}
\item[(1)] $v(g_1 + g_2) \geq \minÊ\{v(g_1), v(g_2)\}$,
\item[(2)] $v(g_1g_2)=v(g_1)+v(g_2)$,
\item[(3)] $v(c)=0$ for a non-zero constant $c \in \calO_{X,0}$.
\end{itemize}
If a function $v : \calO_{X,0} \to \ZZ_{\geq 0} \cup \{ \infty \}$ possesses the properties (1) and (3) but in general not the property (2), it is called an {\em order function}.

A family $\{ v_1, \ldots , v_s \}$ of order functions on the ring $\calO_{X,0}$ defines a multi-index filtration of the ring $\calO_{X,0}$. For $g \in \calO_{X,0}$, let
$$\vv(g):= (v_1(g), \ldots ,  v_s(g)) \in (\ZZ_{\geq 0} \cup \{ \infty \})^s.$$
For $\vv = (v_1, \ldots , v_s) \in \ZZ^s$ the corresponding subspace is defined by 
$$J(\vv) = \{ g \in \calO_{X,0} \, : \, \vv(g) \geq \vv \}.$$
(Here $\vv(g) \geq \vv$ means that $v_i(g)\geq v_i$ for all $i=1, \ldots , s$.)

The notion of the Poincar\'e series of the multi-index filtration $\{ J(\vv) \}$ defined by a family $\{ v_i \}$ of order functions was given in \cite{CDK}. For $\vv \in \ZZ^s$, let $d(\vv)=\dim J(\vv)/J(\vv+\uone)$ where $\uone=(1, 1, \ldots, 1)$. Let
$$
L(\ttt)=\sum\limits_{\vv\in\ZZ^s}d(\vv)\, \ttt^{\,\vv}\,,
$$
where $\ttt=(t_1, \ldots, t_s)$, $\ttt^{\,\vv}=t_1^{v_1}\cdot\ldots\cdot t_s^{v_s}$. (Pay attention that the sum is over all $\vv$ from $\ZZ^s$, not from $\ZZ_{\ge 0}^s$. For $s>1$, the series $L(\ttt)$ contains monomials with negative exponents.) The {\em Poincar\'e series} of the multi-index filtration $\{J(\vv)\}$ is the power series in $\ttt=(t_1, \ldots, t_s)$ defined by
\begin{equation}\label{Ps}
P_{\{v_i\}}(\ttt)=\frac{L(\ttt)\cdot\prod\limits_{i=1}^s(t_i-1)}{t_1\cdot\ldots\cdot t_s-1}\,.
\end{equation}
(This makes sense if all the dimensions $d(\vv)$ are finite.)

The equation (\ref{Ps}) implies that the coefficient at $\ttt^{\,\vv}$ in the Poincar\'e series $P_{\{v_i\}}(\ttt)$ is equal to
\begin{equation}\label{alt}
\sum_{I\subset I_0}(-1)^{\#I}\dim J(\vv+\uone_I)/J(\vv+\uone)\,,
\end{equation}
where $I_0=\{1, 2, \ldots, s\}$, $\uone_I$ is the $s$-tuple the $i$-th component of which is equal to 0 for $i\notin I$ and is equal to 1 otherwise.

\begin{remark}
One can easily see that, if all the subspaces $J(\vv+\uone_{\{i\}})$ (and therefore all the subspaces $J(\vv+\uone_I)$ for $I\ne \emptyset$) are contained in one of them, say, in $J(\vv+\uone_{\{1\}})$, this coefficient is equal to $\dim J(\vv)/J(\vv+\uone_{\{1\}})$.
\end{remark}

The equation (\ref{alt}) (together with the inclusion-exclusion formula) implies that the coefficient at $\ttt^{\,\vv}$ in the Poincar\'e series is equal to the Euler characteristic $\chi(\PP F_{\vv})$ of the projectivisation $\PP F_{\vv}=F_{\vv}/\CC^*$ of the space
$$
F_{\vv}= \left( J(\vv)/J(\vv+\uone)\right) \setminus \bigcup_{i=1}^s \left( J(\vv+\uone_{\{i\}})/J(\vv+\uone) \right)\,
$$
(see \cite{CDG2}).

To compute the Euler characteristic $\chi(\PP F_{\vv})$, it can be convenient to define a $\CC^*$-action on the space $\PP F_{\vv}$ (or on elements of a constructible partitioning of it). In this case the Euler characteristic of the total space coincides with the Euler characteristic of the set of fixed points. In particular, if the $\CC^*$-action is free, the Euler characteristic $\chi(\PP F_{\vv})$ is equal to zero.

One says that a multi-index filtration $\{J(\vv)\}$ on the ring $\calO_{X,0}$ is induced by a (multi-)grading if there exist subspaces $A_{\,\vv}\subset \calO_{X,0}$, $\vv\in\ZZ_{\ge 0}^s$, such that the ring $\calO_{X,0}$ is a completion of the graded algebra $\bigoplus\limits_{\vv\in\ZZ_{\ge 0}^s}A_{\,\vv}$ and $J(\vv)$ is the corresponding completion of $\bigoplus\limits_{\vv'\ge\vv}A_{\,\vv}$. One can easily see that, if a filtration $\{J(\vv)\}$ is induced by a grading $\{A_{\,\vv}\}$ with finite-dimensional subspaces $A_{\,\vv}$, the Poincar\'e series of the filtration $\{J(\vv)\}$ is given by the equation
$$
P(\ttt)=\sum\limits_{\vv\in\ZZ_{\ge0}^s}\dim A_{\,\vv}\cdot \ttt^{\,\vv}\,.
$$
This is not the case in general. Coefficients of the Poincar\'e series are not, generally speaking, dimensions of some spaces. They may be negative (see, e.g., Example~1 at the end of the paper).

\section{Multi-index filtration corresponding to a Newton diagram}
Let $f :(\CC^n,0) \to (\CC,0)$ be the germ of a holomorphic function with an isolated critical point at the origin, non-degenerate with respect to its Newton diagram $\Gamma = \Gamma(f)$ \cite{AGV}. Let $V=\{ f=0\}$ be the corresponding hypersurface singularity. Here we shall define a multi-index filtration on the ring  $\calO_{V,0}$ of germs of functions on the hypersurface $(V,0)$. This filtration is a generalization of the  quasi-homogeneous filtrations defined by the equations of the facets of the Newton diagram.

Suppose that the Newton diagram $\Gamma$ has $s$ faces $\gamma_1, \ldots , \gamma_s$ of dimension $n-1$ (facets), and let $\gamma_i$ lie in the hyperplane given by the equation
$$\ell_i(k_1, \ldots , k_n) = a_1^{(i)} k_1+ \cdots + a_n^{(i)} k_n = d^{(i)}$$
where $a_1^{(i)}, \ldots , a_n^{(i)}$ and $d^{(i)}$ are positive integers with greatest common divisor equal to 1.

For a monomial $\xx^{\,\kk}=x_1^{k_1}\cdot\ldots\cdot x_n^{k_n}$, let
$$
u_i(\xx^{\,\kk}):=\ell_i(k_1, \ldots , k_n)=\sum\limits_{j=1}^n a_j^{(i)} k_j\,,
$$
where $\ell_i(\kk)=\sum\limits_{j=1}^n a_j^{(i)} k_j=d^{(i)}$ is the equation of the face $\gamma_i$ of the Newton diagram $\Gamma$. For a germ $g(x_1, \cdots, x_n)=\sum\limits_{\kk\in\ZZ^n_{\ge0}}c_{\kk}\,\xx^{\,\kk}\in \calO_{\CC^n,0}$, let $u_i(g):=\min\limits_{\kk: c_{\kk}\ne0} u_i(\xx^{\,\kk})$. 
The function $u_i$ is a valuation on the ring $\calO_{\CC^n,0}$. 

\begin{proposition}
The Poincar\'e series $P_{\{u_i\}}(\ttt)$ of the family $\{u_i\}$ of valuations is given by the equation
$$
P_{\{u_i\}}(\ttt)=\prod\limits_{j=1}^n (1-\ttt^{\,\uu(x_j)})^{-1}
$$
$($$x_j$ is the $j$-th coordinate function on the space $\CC^n$$)$.
\end{proposition}

The proof easily follows from the fact that, in this case, the filtration $\{J(\uu)\}$ is induced by a grading. The corresponding subspace $A_{\,\uu}$, $\uu=(u_1, \ldots, u_s)\in\ZZ^s_{\ge 0}$, is generated by the monomials $\xx^{\,\kk}$ with $\ell_i(\kk)=u_i$, $i=1, \ldots, s$.

For a function $g\in \calO_{V,0}=\calO_{\CC^n,0}/(f)$, let
$$
v_i(g):=\max\limits_{g': g'\equiv g \mod f} u_i(g')\,.
$$
The function $v_i$ on the ring $\calO_{V,0}$ is not, generally speaking, a valuation. 
For example, for $f(x,y)=x^5+x^2y^2+y^5$ and for the face of the Newton diagram given by the equation $\ell(k_x,k_y)=2k_x+3k_y$, one has $u(x^2)=4$, $u(x^3+y^2)=6$, but $u(x^5+x^2y^2)=15$.
However, it is an order function. In this way one gets a family $\{v_1, \ldots, v_s\}$ of order functions and the corresponding $s$-index filtration on the ring $\calO_{V,0}$. We shall call it the {\em Newton filtration}.

\section{Poincar\'e series of the Newton filtration for curve singularities}
Let $f$ be the germ of a function of two variables as above and let $\Gamma$ be the corresponding Newton diagram. Let $v_1, \ldots, v_s$ be the order functions on the ring $\calO_{V,0}$ ($V=\{f=0\}$) corresponding to the one-dimensional faces $\gamma_1, \ldots, \gamma_s$ of the diagram $\Gamma$. These order functions are induced by the valuations $u_1$, \ldots, $u_s$ on the ring $\calO_{\CC^2,0}$.

\begin{theorem} \label{Thm:1}
One has
$$
P_{\{v_i\}}(\ttt)=\left\{ \begin{array}{cl}
(1-\ttt^{\,\uu(f)})\cdot P_{\{u_i\}}(\ttt)  & \mbox{ for }s=2,\\
P_{\{u_i\}}(\ttt)  & \mbox{ for }s>2.
\end{array} \right.
$$
\end{theorem}

\begin{proof}
Let the Newton diagram $\Gamma$ consist of two faces (i.e., $s=2$), and let $\mm=(m_1, m_2)$ be the intersection point of them. (The coordinates $m_1$ and $m_2$ are integers.) To compute the coefficient at $\ttt^{\,\vv}$, $\vv=(v_1,v_2)\in\ZZ_{\ge0}^2$, consider the lines $L_i=\{\ell_i(\kk)=v_i\}$, $i=1,2$, and let $\kk=(k_1, k_2)$ be their intersection point.

Suppose that the intersection point of the lines $L_1$ and $L_2$ is either non-integral, or one of its coordinates is negative, or it satisfies the condition $\kk\ge\mm$ (i.e., $k_i\ge m_i$ for $i=1,2$). In $\calO_{\CC^2,0}$ the space $J(\vv)/J(\vv+\uone)$ is freely generated by the monomials $\xx^{\,\kk}$ whose exponents $\kk$ are the integer points on the boundary of the domain $\{\ell_i(\kk)\ge v_i \mbox{ for }i=1,2\}$ (and thus lie on the lines $L_1$ and $L_2$). Using the relation $f=0$ in $\calO_{V,0}$, one eliminates some monomials (if any) on the lines $L_1$ and $L_2$ starting from the intersection point of these lines. (If the intersection point is integral, it is eliminated.) Let $p_i$ be the number of remaining points (monomials) on the line $L_i$, $i=1,2$. Then, in $\calO_{V,0}$, one has
\begin{eqnarray*}
& & \dim J(\vv)/J(\vv+\uone) =   p_1+p_2,\\
& & \dim J(\vv+\uone_{\{1\}})/J(\vv+\uone)  =  p_2, \quad \dim J(\vv+\uone_{\{2\}})/J(\vv+\uone)  = p_1,
\end{eqnarray*}
and the equation (\ref{alt}) implies that the coefficient at $\ttt^{\,\vv}$ is equal to zero.

Suppose that the intersection point of the lines $L_1$ and $L_2$ is integral, non-negative (i.e., $k_i\ge0$ for $i=1,2$) and satisfies the condition $k_1<m_1$. (The case $k_2<m_2$ is treated in the same way.) 
In this case the relation $f=0$ permits to eliminate points (if any) only on the line $L_2$. In particular, the intersection point of the lines $L_1$ and $L_2$ is not eliminated. As above, let $p_i$ be the number of remaining points on the line $L_i$, $i=1,2$. (The intersection point is counted on both of them.) Then, in $\calO_{V,0}$, one has
\begin{eqnarray*}
& & \dim J(\vv)/J(\vv+\uone)= p_1+p_2-1,\\
& & \dim J(\vv+\uone_{\{1\}})/J(\vv+\uone)= p_2-1, \quad \dim J(\vv+\uone_{\{2\}})/J(\vv+\uone)= p_1-1,
\end{eqnarray*}
and the equation (\ref{alt}) implies that the coefficient at $\ttt^{\,\vv}$ is equal to 1.

Therefore 
$$
P_{\{v_i\}}(\ttt)  =  \sum\limits_{\kk\in\ZZ_{\ge0}^2:\kk\not\ge\mm}\ttt^{\,\uu(\xx^{\,\kk})}=
 (1-\ttt^{\,\uu(\xx^{\,\mm})})\cdot \prod\limits_{i=1}^2(1-\ttt^{\,\uu(x_i)})^{-1} 
 =  (1-\ttt^{\,\uu(f)})\cdot P_{\{u_i\}}(\ttt)\,.
$$

Let the Newton diagram consist of more than 2 faces (i.e., $s>2$). For $\vv=(v_1,\ldots,v_s)\in\ZZ_{\ge0}^s$, let $L_i$ be the line $\{\ell_i(\kk)=v_i\}$, $i=1,\ldots, s$. Suppose first that all the lines $L_i$ have (one) common integral point $\kk\ge\uzero$. The relation $f=0$ permits to eliminate some points (if any) on two of the lines $L_i$ (the extreme ones), but not the intersection point $\kk$ itself. In the space $J(\vv)/J(\vv+\uone)$ each subspace $J(\vv+\uone_{\{i\}})/J(\vv+\uone)$ is contained in the subspace $\{c_{\kk}=0\}$ ($c_{\kk}$ is the coefficient at $\xx^{\,\kk}$ in the power series decomposition of a function) and some of them (in fact all but at most two) coincide with this subspace. This implies that the coefficient at $\ttt^{\,\vv}$ ($\vv=(\ell_1(\kk), \ldots, \ell_s(\kk))$) in the Poincar\'e series is equal to
\begin{equation} \label{eq:thm}
\dim J(\vv)/J(\vv+\uone)-\dim\{c_{\kk}=0\}=1
\end{equation}
(see the Remark in Section~1).

Now suppose that the lines $L_i$ do not have a common point in the non-negative orthant. We shall show that in this case the coefficient at $\ttt^{\,\vv}$ in the Poincar\'e series is equal to zero.
We may suppose that each line $L_i$ intersects the boundary of the domain $B=\{\ell_i(\kk)\ge v_i \mbox{ for }i=1,\ldots,s\}$. Otherwise $J(\vv+\uone_{\{i\}})/J(\vv+\uone)=J(\vv)/J(\vv+\uone)$ and the coefficient at $\ttt^{\,\vv}$ is equal to zero according to the Remark in Section~1.

As written above, in $\calO_{\CC^2,0}$, the factorspace $J(\vv)/J(\vv+\uone)$ is freely generated by the monomials $\xx^{\,\kk}$ with $\kk$ from the boundary of the domain $B$. The relations between these generators in $\calO_{V,0}=\calO_{\CC^2,0}/(f)$ correspond to non-negative integer translations of the Newton diagram $\Gamma$ such that the translate of the diagram is contained in the domain $B$ and intersects the boundary of the domain. Suppose that there is a face $\beta_i=L_i\cap B$ of the domain $B$ (possibly of length zero) which cannot intersect a translate (in the described way) of the Newton diagram. In this case the monomials corresponding to integer points on the face $\beta_i$ do not participate in any relation. Multiplication of the coefficients of all these monomials in the power series decomposition of a function by $\lambda\in\CC^*$ defines a free $\CC^*$-action on the part of the space $\PP F_{\vv}$ where at least one coefficient of a monomial corresponding to a point on the boundary of the domain $B$ outside of the face $\beta_i$ is different from zero. Taking into account the condition that all the lines $L_i$ do not have a common non-negative integer point, one can see that the complement to this part may be non-empty only if the length of the face $\beta_i$ is finite, but not zero, both ends $\kk_1$ and $\kk_2$ of this face are integral and the boundary of the domain $B$ consists of $\beta_i$ and two rays (corresponding to the extreme two lines among $L_j$). (All the other lines (if any) go through the points $\kk_1$ or $\kk_2$.) Moreover, the both coefficients at $\xx^{\,\kk_1}$ and at $\xx^{\,\kk_2}$ in the power series decomposition of a function from the complement under consideration are different from zero. Therefore a free $\CC^*$-action on this part can be defined by multiplying the coefficient at $\xx^{\,\kk_1}$ by $\lambda\in\CC^*$.

If all faces of the domain $B$ intersect translates of the Newton diagram, a vertex $\kk$ of it (in fact any one) lies on a translate of $\Gamma$. This means that the coefficient at $\xx^{\,\kk}$ can be eliminated with the help of the corresponding relation. There are no other relations which include points of the boundary $\partial B$ from both connected components of $\partial B\setminus\{\kk\}$. In this case a free $\CC^*$-action on the space $\PP F_{\vv}$ can be defined by multiplying coefficients at all the monomials from one of the connected components of $\partial B\setminus\{\kk\}$ by $\lambda\in\CC^*$. Combing with Equation~(\ref{eq:thm}), this implies the statement of Theorem~\ref{Thm:1} for $s>2$.
\end{proof}

\section{Poincar\'e series of some singularities of more than 2 variables}
It looks somewhat complicated to get the Poincar\'e series corresponding to an arbitrary Newton diagram. However, a special property of some Newton diagrams permits to compute the Poincar\'e series for them in a uniform way.

\begin{definition}
We say that a Newton diagram is {\it stellar} if all its facets (faces of maximal dimension) have a common vertex.
\end{definition}

\begin{example}
Singularities with stellar Newton diagrams include surface singularities of type $T_{p,q,r}$, suspensions of singularities, and singularities with the Newton diagram consisting of 2 faces, in particular, bimodal singularities.
\end{example}

\begin{theorem}
Let the Newton diagram $\Gamma$ of a germ $f\in\calO_{\CC^n,0}$ be stellar. Then
$$
P_{\{v_i\}}(\ttt)=
(1-\ttt^{\,\uu(f)})\cdot P_{\{u_i\}}(\ttt)\,.
$$
\end{theorem}

\begin{proof}
Let $\mm$ be a vertex of the Newton diagram $\Gamma$ on the intersection of all facets of $\Gamma$. One can easily see that $\vv(f)=\vv(\xx^{\,\mm})$. Therefore, for all points $\kk$ in $\Gamma$, one has $\ell_i(\kk)\ge\ell_i(\mm)$ for $i=1, \ldots, s$. For $\vv=(v_1, \ldots , v_s) \in\ZZ_{\ge0}^s$, let $L_i=\{\ell_i(\kk)=v_i\}$, $i=1,\ldots, s$, be the corresponding affine hyperplanes in $\RR^s$.
In $\calO_{\CC^n,0}$, the factorspace $J(\vv)/J(\vv+\uone)$ is freely generated by the monomials $\xx^{\,\kk}$ with $\kk$ from the boundary of the domain $B=\{\ell_i(\kk)\ge v_i \mbox{ for } i=1,\ldots, s\}$. The relations between these generators in $\calO_{V,0}=\calO_{\CC^2,0}/(f)$ correspond to non-negative integer translations of the Newton diagram $\Gamma$ such that the translate of $\Gamma$ is contained in $B$ and intersects the boundary $\partial B$. For each such translation, the translate $\mm'$ of the vertex $\mm$ lies on $\partial B$ as well. (If the values of all linear functions $\ell_i$ at the point $\mm'$ are greater than $v_i$, this holds for the translates of other points of $\Gamma$ as well.) Let $\lambda(\kk)$ be a generic linear function such that it has different values at different integer points (i.e., it is "irrational") and its value at the vertex $\mm$ is greater than at all other points of the Newton diagram $\Gamma$. Using translates of $\Gamma$ in the order of decreasing values of the function $\lambda$ on the translation vectors, one eliminates all the translates of the vertex $\mm$. In this way one eliminates all the integer points $\kk$ on $\partial B$ with $\kk\ge\mm$. The monomials corresponding to the remaining integral non-negative points on $\partial B$ form a basis of the factor space $J(\vv)/J(\vv+\uone)$ in $\calO_{V,0}$. Moreover the space $J(\vv+\uone_{\{i\}})/J(\vv+\uone)$ is freely generated by the monomials corresponding to those points which do not lie on $L_i$. The equation (\ref{alt}) and the inclusion-exclusion formula imply that the coefficient at $\ttt^{\,\,\vv}$ in the Poincar\'e series is equal to the number of those non-negative integer points $\kk\in\partial B$ with $\kk\not\ge\mm$ which belong to all the hyperplanes $L_i$. Therefore
$$
P_{\{v_i\}}(\ttt)=\sum\limits_{\kk\in\ZZ_{\ge0}^s:\kk\not\ge\mm}\ttt^{\,\uu(\xx^{\,\kk})}
=(1-\ttt^{\,\uu(\xx^{\,\mm})})\cdot \prod\limits_{i=1}^s(1-\ttt^{\,\uu(x_i)})^{-1}=(1-\ttt^{\,\,\uu(f)})\cdot P_{\{u_i\}}(\ttt)\,.
$$
\end{proof}

Up to now we had only two types of equations for the Poincar\'e series of Newton filtrations like in Theorem~\ref{Thm:1}. Moreover all coefficients in these series were non-negative. This could produce a hope that Newton filtrations are induced by gradings of the coordinate ring. The following examples show that, in general, neither all the Poincar\'e series of Newton filtrations are of these types nor all the coefficients in them are non-negative. In these examples we compute the coefficient of the Poincar\'e series at $\ttt^{\, \vv(f)}$.

\begin{examples} 
{\bf 1.} For $f(x,y,z)=x^5+y^5+z^5+x^2yz+xy^2z+xyz^2$ the Newton diagram consists of 4 facets $\gamma_0$, $\gamma_1$, $\gamma_2$, and $\gamma_3$ lying on the hyperplanes with the equations $k_x+k_y+k_z=4$, $2k_x+k_y+k_z=5$, $k_x+2k_y+k_z=5$, and $k_x+k_y+2k_z=5$ respectively. Besides the vertices $(2,1,1), \ldots , (0,0,5)$ of the Newton diagram, there are 12 integral points on the diagram: 4 on each of the facets $\gamma_1$, $\gamma_2$, and $\gamma_3$. One integer point, say $(0,0,5)$ can be eliminated using the relation $f=0$. The space $J(\vv)/J(\vv+\uone)$ ($\vv=\vv(f)$) is freely generated by the 17 remaining monomials. The subspace $F_{\vv}$ in it is the complement of the union of the 4 subspaces given by the equations ($c_{\kk}$ is the coefficient at $\xx^{\, \kk}$ in the power series decomposition of a function):
\begin{eqnarray*}
J(\vv+\uone_{\{0 \}})/J(\vv+\uone) & = & \{ a_{211} = a_{121} = a_{112} \}, \\
J(\vv+\uone_{\{1 \}})/J(\vv+\uone) & = & \{ a_{050} = a_{041} = a_{032} = a_{023} = a_{014} = a_{121} = a_{112} =0\}, \\
J(\vv+\uone_{\{2 \}})/J(\vv+\uone) & = & \{ a_{500} = a_{401} = a_{302} = a_{203} = a_{104} = a_{211} = a_{112} =0\}, \\
J(\vv+\uone_{\{3 \}})/J(\vv+\uone) & = & \{ a_{500} = a_{410} = a_{320} = a_{230} = a_{140} = a_{050} = a_{211} = a_{121}\}.
\end{eqnarray*}
From these data one can easily compute that the coefficient at $\ttt^{\, \vv(f)} = t_0^4t_1^5t_2^5t_3^5$ in the Poincar\'e series is equal to $-1$. Therefore, in this case,  the Newton filtration is not induced by a grading.

\noindent {\bf 2.} For 
$f(x,y,z) = x^{20}+y^{20}+z^{16}+ x^8y^8+x^6y^6z^2+x^2y^2z^{10}+x^3y^8z^3+x^8y^3z^3$
the Newton diagram consists of 5 facets $\gamma_0$, $\gamma_1$, $\gamma_2$, $\gamma_3$, and $\gamma_4$ lying on the hyperplanes with the equations $k_x+k_y+k_z=14$, $2k_x+3k_y+5k_z=40$, $3k_x+2k_y+5k_z=40$, $11k_x+4k_y+5k_z=80$, and $4k_x+11k_y+5k_z=80$ respectively. Computations like in Example~1 yield the coefficient at  $\ttt^{\, \vv(f)} = t_0^{14}t_1^{40}t_2^{40}t_3^{80}t_4^{80}$ to be equal to 1. Since $\vv(f)=(14,40,40,80,80)$ is not a linear combination of $\vv(x)=(1,2,3,11,4)$, $\vv(y)=(1,3,2,4,11)$, and $\vv(z)=(1,5,5,5,5)$, the Poincar\'e series is not of one of the types of Theorem~\ref{Thm:1}.
\end{examples}


\bigskip
\noindent Leibniz Universit\"{a}t Hannover, Institut f\"{u}r Algebraische Geometrie,\\
Postfach 6009, D-30060 Hannover, Germany \\
E-mail: ebeling@math.uni-hannover.de\\

\medskip
\noindent Moscow State University, Faculty of Mechanics and Mathematics,\\
Moscow, GSP-1, 119991, Russia\\
E-mail: sabir@mccme.ru
\end{document}